\documentclass[12pt]{amsart}

\usepackage[T1]{fontenc}
\usepackage[utf8]{inputenc}
\usepackage{lmodern}

\usepackage{babel}

\usepackage[letterpaper,margin=2.3cm]{geometry}

\usepackage{xifthen}
\usepackage{ifthenx}

\usepackage{xargs}

\usepackage{graphicx}
\usepackage[table,svgnames,x11names]{xcolor}

\usepackage{amssymb}
\usepackage{amsmath}
\usepackage{amsthm}
\usepackage{mathtools}

\usepackage{exscale}
\usepackage{relsize}
\usepackage{bbm}
\usepackage{bm}
\usepackage{amsbsy}
\usepackage{mathdots}

\usepackage{mathrsfs}

\usepackage{stmaryrd}

\usepackage{mdframed}

\usepackage[all]{xy}

\usepackage{fp}

\usepackage[section]{placeins}
\usepackage{float}

\usepackage{enumerate}

\newcounter{proof}
{\stepcounter{proof}\begin{proof}}%
{\end{proof}}%
\newcounter{proofstep}[proof]
{\refstepcounter{proofstep}\bigskip\par\noindent%
  \ifthenelse{\isempty{#1}}
    {\textit{Step \theproofstep. }}
    {\textit{#1.}}
  \noindent}%
{\par}%
\newcounter{proofstep-noskip}[proof]
\newenvironment{proofstep-noskip}[1][]%
{\refstepcounter{proofstep-noskip}
  \ifnum\the\value{proofstep-noskip}>1
    \bigskip\par\noindent
  \fi
  \ifthenelse{\isempty{#1}}
    {\textit{Step \theproofstep-noskip. }}
    {\textit{#1.}}
  \noindent}%
{\par}%
\newcounter{proofcase}[proof]
{\refstepcounter{proofcase}\bigskip\par\noindent%
  \ifthenelse{\isempty{#1}}
    {\textit{Case \theproofcase. }}
    {\textit{#1.}}
  \noindent}%
{\par}%

\newcounter{proofcase-noskip}[proof]
\newenvironment{proofcase-noskip}[1][]%
{\refstepcounter{proofcase-noskip}
  \ifnum\the\value{proofcase-noskip}>1
    \bigskip\par\noindent
  \fi
  \ifthenelse{\isempty{#1}}
    {\textit{Case \theproofcase-noskip. }}
    {\textit{#1.}}
  \noindent}%
{\par}%

\usepackage{varioref}
\usepackage{hyperref}
\hypersetup{
  colorlinks,
  allcolors=DarkBlue
}
\usepackage[nameinlink]{cleveref}

\theoremstyle{plain}
\newtheorem{thm}{Theorem}[section]
\newtheorem*{thm*}{Theorem}
\newtheorem{pro}[thm]{Proposition}
\newtheorem{cor}[thm]{Corollary}
\newtheorem{lem}[thm]{Lemma}

\theoremstyle{definition}

\theoremstyle{remark}

\numberwithin{equation}{section}

\AddToHook{env/pro/begin}{\crefalias{thm}{pro}} 
\AddToHook{env/cor/begin}{\crefalias{thm}{cor}}
\AddToHook{env/lem/begin}{\crefalias{thm}{lem}}
\AddToHook{env/question/begin}{\crefalias{thm}{question}}
\AddToHook{env/dfn/begin}{\crefalias{thm}{dfn}}
\AddToHook{env/ntn/begin}{\crefalias{thm}{ntn}}
\AddToHook{env/rem/begin}{\crefalias{thm}{rem}}
\AddToHook{env/clm/begin}{\crefalias{thm}{clm}}

\newcommandx{\textref}[2][1=]{\hyperref[#2]{#1\ref*{#2}}}
\newcommandx{\textrefp}[2][1=]{(\hyperref[#2]{#1\ref*{#2}})}

\newcommand{\vvvert}{{\vert\kern-0.25ex\vert\kern-0.25ex\vert}}

\usepackage{tikz}
\usetikzlibrary{trees}
\usepackage{tikz-cd}

\begin{document}

\title[Even and odd Haar levels]%
{Subspaces of $L^1$ spanned by the even and odd levels of the Haar system}%

\author[Th.~Speckhofer]{Thomas Speckhofer}%
\address{Thomas Speckhofer, Department of Mathematics, Texas A\&M University, College Station, TX~77843, USA}
\email{speckhofer@tamu.edu}

\date{\today}%



\thanks{The author was supported by the Austrian Science Fund~(FWF), project 10.55776/J5020.}

\begin{abstract}
  We prove that the closed subspace of~$L^1[0,1]$ spanned by the even (respectively, odd) levels of the Haar system is not isomorphic to~$L^1[0,1]$.
  Moreover, the closed span of any infinite collection of Haar levels with infinite complement is isomorphic to that of the even levels.
\end{abstract}

\maketitle


\section{Introduction}
\label{sec:introduction}

The Haar system and its properties in function spaces such as $L^p = L^p[0,1]$, $1\le p < \infty$, have been studied extensively throughout the last decades. A fundamental problem is to determine the isomorphic types of spaces generated by subsequences of the Haar system, i.e., of the closed linear span $[h_I : I\in \mathcal{A}]$, where $h_I$ is the Haar function supported on a dyadic interval $I\subset [0,1]$ and $\mathcal{A}$ is an infinite collection of dyadic intervals. In~1973, Gamlen and Gaudet~\cite{MR0328575} obtained a complete characterization of these isomorphic types in the case of $L^p$, $1 < p < \infty$, where the Haar system is unconditional: Consider $\limsup(\mathcal{A})$, the set of points contained in infinitely many intervals $I\in \mathcal{A}$. If $\limsup(\mathcal{A})$ has positive measure, then the subspace $[h_I : I\in \mathcal{A}]$ is isomorphic to~$L^p$, and otherwise, it is isomorphic to~$\ell^p$. The above classification problem has also been solved for the dyadic Hardy space~$H^1$ by Müller~\cite{MR879418} (see also~\cite{MR2157745}). There, the only isomorphic types that occur are $H^1$, $\ell^1$, and $(\sum H^1_n)_{\ell^1}$, where $H^1_n$ denotes the span in $H^1$ of the Haar system up to level~$n$.

If $\limsup(\mathcal{A})$ has positive measure, then the proof that $[h_I : I\in \mathcal{A}]$ is isomorphic to the underlying function space $X = L^p$, $1 < p < \infty$, or $H^1$, can be outlined as follows. One constructs a block basis of the Haar system in $[h_I : I\in \mathcal{A}]$ that is equivalent to the original Haar system and spans a complemented copy of~$X$. Since the Haar system is unconditional in the aforementioned spaces, $[h_I : I\in \mathcal{A}]$ itself is also complemented in~$X$, so by Pe{\l}czy\'{n}ski's decomposition method~\cite{MR126145}, the two spaces are isomorphic. If the underlying space is~$L^1$, then it is still possible to construct a complemented copy of~$L^1$ using block bases, but the decomposition method cannot be applied since the space $[h_I : I\in \mathcal{A}]$ need not be complemented in~$L^1$. For example, it is known to be uncomplemented if $\mathcal{A}$ consists of the even (or odd) levels of the Haar system (see~\cite{MR2224004}; see also~\cite{MR4098600}). In fact, it has been a long-standing open problem whether the closed subspace of~$L^1$ spanned by the even (or odd) levels of the Haar system is isomorphic to~$L^1$. This problem is recorded in~\cite{MR2224004}, where it is attributed to Schechtman.
Independently, this problem was also posed by Dilworth and Kutzarova in the context of quasi-greedy basic sequences in~$L^1$~\cite{Kutzarova2016Comments}. Moreover, a finite-dimensional version was formulated by Dilworth, Kutzarova and Ostrovskii~\cite[Problem~6.8]{MR4098600}, who asked whether the cycle spaces $Z(D_n)$ of the diamond graphs~$D_n$ are uniformly isomorphic to the spaces $\ell^1_{k(n)}$ of the same dimension (note that $Z(D_n)$ is isometric to the span of the first $n$ even levels of the Haar system in~$L^1$).

In this paper, we give a negative answer to the above problem. Our main result is the following.

\begin{thm}\label{thm:main-result}
  Let $X_0$ and $X_1$ denote the closed linear spans of the even, respectively odd, levels of the Haar system in~$L^1$. Then $X_0$ and $X_1$ are not isomorphic to~$L^1$.
\end{thm}
Note that $X_0$ and~$X_1$ cannot be isomorphic to $\ell^1$ either, since they fail the Schur property.
We will prove \Cref{thm:main-result} in \Cref{sec:proof}. 
The proof is based on the fact that every bounded linear operator from~$L^1$ into~$\ell^2$ is absolutely $2$-summing, which is a consequence of the Grothendieck inequality~\cite{MR94682} (see \cite[Theorem~4.1]{MR231188}). We explicitly construct a bounded linear operator from $X_0$ into $\ell^2$ which is not absolutely $2$-summing; thus, $X_0$ cannot be isomorphic to~$L^1$.
The proof also yields a negative answer to the finite-dimensional problem: the Banach-Mazur distance between $Z(D_n)$ and $\ell^1_{k(n)}$ grows at least on the order of~$\sqrt{n}$.
Furthermore, in \Cref{sec:X0-X1-isomorphic}, we show that the subspace~$X_0$ is isomorphic to~$X_1$; hence, $X_1$ is not isomorphic to~$L^1$ either.
The same is true for all subspaces of $L^1$ generated by an infinite set of Haar levels $S\subset \mathbb{N}_0$ whose complement in $\mathbb{N}_0$ is also infinite (see \Cref{cor:infinite-coinfinite-levels}).

Before presenting the proof of \Cref{thm:main-result}, we establish some notation and terminology.
Let $\mathbb{N}_0 = \mathbb{N}\cup \{ 0 \}$. For $n\in\mathbb{N}_0$, we let
\begin{equation*}
\mathcal{D}^n = \Big\{\Big[\frac{i-1}{2^n},\frac{i}{2^n}\Big):1\leq i\leq 2^n\Big\}\qquad \text{ and }\qquad \mathcal{D} = \bigcup_{n=0}^\infty\mathcal{D}^n,\quad \mathcal{D}_{\mathrm{even}} = \bigcup_{\substack{n=0\\ n\text{ even}}}^{\infty} \mathcal{D}^n.
\end{equation*}
We enumerate the collection $\mathcal{D}$ of dyadic intervals according to the lexicographical order, defined by $I<J$ whenever $|I|>|J|$, or $|I|=|J|$ and $\sup(I)\le \inf(J)$.  For each dyadic interval $I\in \mathcal{D}$, let~$I^+$ denote the left half of~$I$ and~$I^-$ its right half (both are again in $\mathcal{D}$).
The Haar system $(h_I)_{I\in\mathcal{D}}$ is defined by
\begin{equation*}
  h_I
  = \chi_{I^+} - \chi_{I^-},
  \qquad I\in\mathcal{D},
\end{equation*}
where $\chi_A$ denotes the characteristic function of a set $A$.
Together with the constant function $\chi_{[0,1)}$, the Haar system in its lexicographic order is a monotone Schauder basis
of $L^p$, $1\le p<\infty$, and it is unconditional if and only if $1<p<\infty$. A diagonal operator with respect to the Haar system in $L^p$ is called a Haar multiplier. We denote the closed unit ball of a Banach space~$X$ by $B_X$. For a subset $F$ of a Banach space $X$, we denote by $[F]$ its closed linear span in~$X$.

\section{Proof of the main result}
\label{sec:proof}

\begin{proof}[Proof of \Cref{thm:main-result}]
  Since $X_0$ is isomorphic to $X_1$ (see \Cref{pro:X0-isomorphic-X1}), it suffices to prove the result for~$X_0$.
  First, we will show that
  \begin{equation*}
    T\colon X_0 \to \ell^2(\mathcal{D}_{\mathrm{even}}),  \qquad Tf = \Bigl( \int f h_I \Bigr)_{I\in \mathcal{D}_{\mathrm{even}}}
  \end{equation*}
  is a bounded linear operator with $\|T\|\le \sqrt{2}$.

  Fix $f\in X_0$. For an even-level dyadic interval $I\in \mathcal{D}_{\mathrm{even}}$, let $I_1,\dots,I_4$ denote its four consecutive subintervals of length $|I|/4$. Since all odd-level Haar coefficients of~$f$ are zero, there exist scalars $\alpha,\beta$ such that
  \begin{equation*}
    \int_{I_1} f = \int_{I_2} f = \alpha\qquad \text{and} \qquad \int_{I_3} f = \int_{I_4} f = \beta.
  \end{equation*}
  Write $a_J = \int_J |f|$ for $J\in \mathcal{D}$. Then
  \begin{align*}
    a_I^2 - \sum_{j=1}^4 a_{I_j}^2 &= 2 \sum_{1\le j < k \le 4} a_{I_j}a_{I_k}\\
                                   &\ge 2\bigl(|\alpha|^2 + |\beta|^2 + 4|\alpha| |\beta|\bigr)\\
    &\ge 2 |\alpha - \beta|^2 = \frac{1}{2} \Bigl| \int f h_I \Bigr|^2.
  \end{align*}
  We can now apply the above inequality to $\|Tf\|_{\ell^2(\mathcal{D}_{{\mathrm{even}}})}$ to obtain a telescoping sum: For every $N\in \mathbb{N}_0$, we have
  \begin{equation*}
    \sum_{n=0}^N \sum_{I\in \mathcal{D}^{2n}} \Bigl| \int f h_I \Bigr|^2
    \le 2 \Bigl( \|f\|_{L^1}^2 - \sum_{I\in \mathcal{D}^{2N+2}} a_I^2 \Bigr)
    \le 2 \|f\|_{L^1}^2.
  \end{equation*}
  Letting $N\to \infty$, we obtain $\|T\|\le \sqrt{2}$.

  It follows from the Grothendieck inequality~\cite{MR94682} that every bounded linear operator from~$L^1$ into~$\ell^2$ is absolutely $1$-summing and hence absolutely $2$-summing (see \cite[Theorem~4.1]{MR231188}). We prove that $T$ is not absolutely $2$-summing; hence, its domain~$X_0$ cannot be isomorphic to~$L^1$.

  For $I\in \mathcal{D}_{\mathrm{even}}$, let
  \begin{equation*}
    e_I = |I|^{-1/2}h_I\in X_0.
  \end{equation*}
  Thus, $(e_I)_{I\in \mathcal{D}_{\mathrm{even}}}$ is an orthonormal sequence in~$L^2$.
  For every finite subcollection $\mathcal{C}\subset \mathcal{D}_{\mathrm{even}}$, we claim that
  \begin{equation}\label{eq:1}
    \sup_{x^{*}\in B_{X_0^{*}}} \Bigl( \sum_{I\in \mathcal{C}} |x^{*}(e_I)|^2 \Bigr)^{1/2}\le 1.
  \end{equation}
  Indeed, by Hahn-Banach, any functional $x^{*}\in B_{X_0^{*}}$ can be represented by some function $g\in L^{\infty}$ with $\|g\|_{L^{\infty}}\le 1$, and by Bessel's inequality, the expression inside the supremum in~\eqref{eq:1} is bounded by $\|g\|_{L^2}\le \|g\|_{L^{\infty}}\le 1$.

  On the other hand, we have
  \begin{equation*}
    Te_I = |I|^{1/2} \delta_I,
  \end{equation*}
  where $\delta_I$ denotes the standard basis vector in~$\ell^2(\mathcal{D}_{\mathrm{even}})$ supported at $I$. Hence, for $N\in \mathbb{N}_0$ and $\mathcal{C} = \bigcup_{n=0}^N \mathcal{D}^{2n}$, we obtain
  \begin{equation*}
    \sum_{I\in \mathcal{C}} \|T e_I\|_{\ell^2(\mathcal{D}_{\mathrm{even}})}^2 = \sum_{n=0}^N \sum_{I\in \mathcal{D}^{2n}} |I| = N + 1.
  \end{equation*}
  Together with~\eqref{eq:1}, this implies that $T$ is not absolutely $2$-summing.
\end{proof}

\section{$X_0$ and $X_1$ are isomorphic}
\label{sec:X0-X1-isomorphic}

  For completeness, we provide a proof of the fact that the two subspaces $X_0$ and $X_1$ of~$L^1$ are isomorphic. To the author's knowledge, this has not previously appeared in the literature. We will use the following elementary observation.

  \begin{lem}\label{lem:haar-multiplier}
    Let $d_I\in \{ 0,1 \}$ for each $I\in \mathcal{D}$, and assume that whenever $d_I = 0$, we also have $d_{I^+} = d_{I^-} = 0$. Then the linear extension of $Dh_I = d_I h_I$, $I\in \mathcal{D}$, defines a bounded Haar multiplier on $L^1_0 := \{ f\in L^1 : \int f = 0 \}$ with $\|D\|\le 1$.
  \end{lem}
  \begin{proof}
    This can be seen by repeatedly applying the following fact: If $I\in \mathcal{D}$, $a_I$ is a scalar, and $f\in L^1_0$ is constant on~$I$, then $f + a_I h_I$ and $f - a_I h_I$ have the same distribution, so
    \begin{equation*}
       \|f\|_{L^1} \le \frac{1}{2}\bigl(\|f + a_I h_I\|_{L^1} + \|f - a_Ih_I\|_{L^1}\bigr) = \|f + a_Ih_I\|_{L^1}.\qedhere
    \end{equation*}
  \end{proof}
  
  \begin{pro}\label{pro:X0-isomorphic-X1}
    The space $X_0$ is isomorphic to $X_1$ and to $\ell^1(X_0)$.
  \end{pro}
  \begin{proof}
    By splitting the unit interval $[0,1]$ into its halves and restricting any function in $X_1$ to the subintervals $[0,1/2)$ and $[1/2,1)$, we see that $X_1 \simeq X_0 \oplus X_0$. Analogously, we have $X_0\simeq \langle\{ h_{[0,1)} \}  \rangle \oplus X_1\oplus X_1$. Thus, it suffices to prove that $X_0\simeq \ell^1(X_0)$; by Pe{\l}czy\'{n}ski's decomposition method~\cite{MR126145}, it then follows that $X_0\simeq X_1$.

    To this end, put $I^{(n)} = [0,4^{-n})$ for all $n\in \mathbb{N}_0$. Then $Y := [h_{I^{(n)}} : n\in \mathbb{N}_0]$ is complemented in $X_0$ by a Haar multiplier of the form as in \Cref{lem:haar-multiplier} (with $d_I = 1$ if $\inf I = 0$ and $d_I = 0$ otherwise), restricted to $X_0$. Moreover, a straightforward computation shows that $(4^nh_{I^{(n)}})_{n=0}^{\infty}$ is equivalent to the unit vector basis of $\ell^1$, so $Y \simeq \ell^1$. On the other hand, consider the complement
    \begin{equation*}
      Z = \bigl[h_L : L\in \mathcal{D}_{\mathrm{even}} \setminus \{ I^{(n)} : n\in \mathbb{N}_0 \}\bigr].
    \end{equation*}
    For every $n\in \mathbb{N}_0$, the set $I^{(n)}$ can be written as a disjoint union of its four quarters $I^{(n)} = I^{(n+1)}\cup I^{(n)}_2 \cup I^{(n)}_3 \cup I^{(n)}_4$. The norm of any function $f\in Z$ satisfies
    \begin{equation*}
      \|f\|_{L^1} = \sum_{n=0}^{\infty} \sum_{k=2}^4 \|f\chi_{I^{(n)}_k}\|_{L^1},
    \end{equation*}
    and for every $n\in \mathbb{N}_0$ and $k = 2,3,4$, the closed span $[h_L : L\in \mathcal{D}_{\mathrm{even}},\, L\subset I^{(n)}_k]$ is isometrically isomorphic to $X_0$ via the usual affine dilation from $I^{(n)}_k$ onto $[0,1)$. Thus, we conclude that $Z \simeq \ell^1(X_0)$, which implies that $X_0\simeq \ell^1\oplus \ell^1(X_0)$ and hence $X_0 \simeq \ell^1(X_0)$.
  \end{proof}

    Finally, for a subset $S\subset \mathbb{N}_0$, let $X_S = [h_I : |I| = 2^{-n},\ n\in S]\subset L^1$. Note that $X_S\simeq L^1$ if the complement of $S$ is finite. On the other hand, we will show that $X_S\simeq X_0$ whenever both $S$ and its complement are infinite. To prove this, we will use the following lemma.
  \begin{lem}\label{lem:levels-embedding}
    Let $S,T\subset \mathbb{N}_0$, and suppose that $\phi\colon \mathbb{N}_0\to \mathbb{N}_0$ is a strictly increasing function with $n\in S\iff \phi(n)\in T$. Then $X_S$ is isometric to a $1$-complemented subspace of $X_T$.
  \end{lem}
  \begin{proof}
    For every non-empty finite subset $A\subset \mathbb{N}_0$, define the Walsh function
    \begin{equation*}
      w_A = \prod_{k\in A} r_k,
    \end{equation*}
    where $r_k = \sum_{I\in \mathcal{D}^k} h_I$ is the $k$th Rademacher function. For every $n\in \mathbb{N}_0$, the span of the $n$th level of the Haar system coincides with $[w_A : \max A = n]$. Hence, $X_S = [w_A : \max A \in S]$. Consider the map $J_\phi$ on the Walsh system  that replaces each Rademacher function $r_k$ by $r_{\phi(k)}$, i.e., $J_\phi w_A = w_{\phi(A)}$. This substitution preserves the joint distribution of every finite collection of Rademacher functions, hence, by density, it extends to a linear isometry $J_\phi\colon L^1_0\to L^1_0$. Since $\max \phi(A) = \phi(\max A)$ for all non-empty finite $A\subset \mathbb{N}_0$, we see that $J_\phi(X_S)\subset X_T$. Now consider the $\sigma$-algebra $\mathcal{F}_\phi = \sigma(\{ r_{\phi(k)} : k\in \mathbb{N}_0 \})$. For every non-empty finite $B\subset \mathbb{N}_0$, the conditional expectation $\mathbb{E}^{\mathcal{F}_\phi}$ satisfies
    \begin{equation*}
      \mathbb{E}^{\mathcal{F}_\phi}w_B =
      \begin{cases}
        w_B, & B\subset \phi(\mathbb{N}_0),\\
        0, & \text{otherwise.}
      \end{cases}
    \end{equation*}
    Thus, $\mathbb{E}^{\mathcal{F}_\phi}(X_T) = [w_{\phi(A)} : \phi(\max A)\in T] = [w_{\phi(A)} : \max A\in S] = J_\phi(X_S)$.
  \end{proof}

  Together with \Cref{pro:X0-isomorphic-X1} and Pe{\l}czy\'{n}ski's decomposition method, \Cref{lem:levels-embedding} implies the following result.
  \begin{cor}\label{cor:infinite-coinfinite-levels}
    Let $S\subset \mathbb{N}_0$, and assume that both $S$ and $\mathbb{N}_0\setminus S$ are infinite. Then $X_S$ is isomorphic to~$X_0$.
  \end{cor}

\noindent\textbf{Lean formalization.}
The proofs of \Cref{thm:main-result} and \Cref{cor:infinite-coinfinite-levels}
have been formalized in Lean~4 with Mathlib.
The source code and instructions for checking the formalization
are available at
\begin{center}
  \url{https://github.com/speckhofer/even-odd-haar-levels-lean}
\end{center}
\medskip
    
\noindent\textbf{Acknowledgments.} The author would like to thank Pavlos Motakis and Thomas Schlumprecht for valuable comments and discussions. Moreover, he is grateful to Stephen Dilworth and Denka Kutzarova for helpful remarks, as well as to William B.\ Johnson for fruitful discussions, which led to the inclusion of \Cref{cor:infinite-coinfinite-levels}.
\medskip

\noindent\textbf{AI usage statement.}
The author used OpenAI's ChatGPT~6.0 Astra in developing the results presented in this paper. In particular, the proof of the main theorem was obtained in a conversation with this model, in which it was prompted to find an isomorphic invariant that can be used to establish the stated result. The author has checked the mathematical correctness of the proof, edited it for clarity, presentation, and context, and takes full responsibility for the mathematical content of the paper.
GPT-6 Astra and GPT-6.1 Sol were used for the accompanying Lean formalization.

\bibliographystyle{abbrv}%
\bibliography{bibliography}%

\end{document}